\documentclass[reqno,12pt]{amsart}
\usepackage{amsmath,amsthm,amsfonts,amssymb}

\newtheorem{theorem}{Theorem}[section]

\newtheorem{lemma}[theorem]{Lemma}

\newtheorem{proposition}[theorem]{Proposition}

\newtheorem{remark}[theorem]{Remark}
\numberwithin{equation}{section}

\begin{document}

 \centerline{{\bf Log-Sobolev-type inequalities for solutions to stationary}}

  \centerline{{\bf Fokker--Planck--Kolmogorov equations}}

  \vskip .1in

  \centerline{V.I. Bogachev$^{a,b,c}$,  A.V. Shaposhnikov$^{a}$,  S.V. Shaposhnikov$^{a,b}$}

\vskip .1in

\noindent
$^{a}$ Department of Mechanics and Mathematics, Moscow State University, 119991 Moscow, Russia

\noindent
$^{b}$ National Research University Higher School of Economics, ul. Usacheva 6, 119048 Moscow, Russia

\noindent
$^{c}$ The corresponding author, vibogach@mail.ru

 \vskip .1in

 {\bf Abstract.} We prove that every probability measure $\mu$ satisfying
 the stationary Fokker--Planck--Kolmogorov equation obtained by a $\mu$-integrable
 perturbation $v$ of the drift term $-x$ of the Ornstein--Uhlenbeck operator
 is absolutely continuous with respect to the corresponding Gaussian measure $\gamma$
 and for the density $f=d\mu/d\gamma$
 the integral of $f |\log (f+1)|^\alpha$ against $\gamma$ is estimated
 via $\|v\|_{L^1(\mu)}$ for all $\alpha<1/4$, which is a weakened $L^1$-analog of the logarithmic
 Sobolev inequality. This yields that stationary measures of infinite-dimensional
 diffusions whose
  drifts are integrable perturbations of $-x$ are absolutely continuous with respect
  to Gaussian measures. A generalization is obtained for
  equations on Riemannian manifolds.

\vskip .1in


\vskip .1in


\section{Introduction}

 It is known (see \cite{BKR01} and \cite{BKRS})
  that a probability measure $\mu$ on $\mathbb{R}^d$ satisfying the
 stationary Fokker--Planck--Kolmogorov equation
 \begin{equation}\label{eq1}
 \Delta\mu -{\rm div}\, (b\mu)=0
 \end{equation}
 in the sense of the integral identity
 $$
 \int_{\mathbb{R}^d} [\Delta \varphi +\langle b,\nabla \varphi\rangle ]\, d\mu=0,
 \quad \varphi\in C_0^\infty(\mathbb{R}^d),
 $$
 where $b\colon\mathbb{R}^d\to \mathbb{R}^d$ is Borel measurable and integrable with respect to $\mu$ on balls,
 possesses a density $\varrho$ with respect to Lebesgue measure. Moreover, if $|b|$ is locally integrable
 to some power $p$ greater than $d$ with respect to $\mu$ or with respect to Lebesgue measure, then
 $\varrho$ belongs to the Sobolev class $W^{p,1}$ on every ball
 (the class $W^{p,1}(\Omega)$ on a domain~$\Omega$ consists of functions belonging
  to $L^p(\Omega)$ along with their generalized first order derivatives).
  However, this is false if $p<d$ (see \cite{BKRS}).
 On the other hand, as shown in~\cite{BR95} (see also~\cite{BKRS}), in the case
 of the global condition
 $|b|\in  L^2(\mu)$, we have $\varrho\in W^{1,1}(\mathbb{R}^d)$,
 $\sqrt{\varrho}\in W^{2,1}(\mathbb{R}^d)$ and
 $$
 \int_{\mathbb{R}^d} \frac{|\nabla \varrho|^2}{\varrho}\, dx\le
 \int_{\mathbb{R}^d} |b|^2\, d\mu.
 $$
 The latter bound admits an infinite-dimensional version.
To this end we write the drift $b$ in the form
 $$
 b(x)=-x+v(x).
 $$
 If $v=0$, then the only solution in the class of probability measures is the standard Gaussian measure $\gamma$
 with density $(2\pi)^{-d/2}\exp (-|x|^2/2)$. Hence it is natural to express $\mu$ through $\gamma$. For the corresponding
 density $f=d\mu/d\gamma$ one has
 $$
 \int_{\mathbb{R}^d} \frac{|\nabla f|^2}{f}\, d\gamma\le
 \int_{\mathbb{R}^d} |v|^2\, d\mu.
 $$
 In this form the result extends to the infinite-dimensional case provided that $v$ takes values
 in the Cameron--Martin space $H$ of the Gaussian measure $\gamma$ and $|v|=|v|_H$ and
 $|\nabla f|=|\nabla f|_H$ are
 taken with respect to the Cameron--Martin norm. The logarithmic Sobolev inequality
 (applied to~$\sqrt{f}$) yields the bound
 $$
 \int_{\mathbb{R}^d} f|\log f|\, d\gamma\le
 \int_{\mathbb{R}^d} |v|^2\, d\mu,
 $$
 as well as its infinite-dimensional analog,
 which is a constructive sufficient condition for the uniform integrability
 of the densities of finite-dimensional projections of solutions to infinite-dimensional equations with
 respect to the corresponding Gaussian measures.

 It has recently been shown in \cite{BPS} that these results on Sobolev differentiability
 of densities break down in the $L^1$-setting.
 It can happen that $|b|\in L^1(\mu)$, but the solution density $\varrho$ does not belong
 to the Sobolev class $W^{1,1}(\mathbb{R}^d)$, i.e., $|\nabla\varrho|$
 does  not belong to~$L^1(\mathbb{R}^d)$, and similarly for the density $f$
 the condition $|v|\in L^1(\gamma)$ does not guarantee that
 the function $|\nabla f|$  belongs to $L^1(\gamma)$. However, these negative results
 left open the important question of whether in the infinite-dimensional case the solution $\mu$ with
 $|v|_H\in L^1(\mu)$ is always absolutely continuous with respect
 to $\gamma$ as it holds in the finite-dimensional case.
 The main result of this paper answers positively this
 long-standing question. This result is based
 on a dimension-free finite-dimensional bound on the integral of $f|\log (f+1)|^\alpha$.
It is worth noting that if $f>0$ is in the Sobolev class $W^{1,1}(\gamma)$ with respect to $\gamma$,
then the measure $f\cdot\gamma$ satisfies (\ref{eq1}) with $v=\nabla f/f$. It is known
(see \cite{Led}, \cite{Am}, and \cite{FH}) that in this very special case
$f\sqrt{\log (f+1)}\in L^1(\gamma)$.

\section{Main results}

Throughout this section we use the notation $\|v\|_{L^1(\mu)}:=\bigl\| |v| \bigr\|_{L^1(\mu)}$.

The following theorem is our main result.

\begin{theorem}\label{mt1}
For every $\alpha<1/4$, there is a number $C(\alpha)$ such that
whenever $\mu=f\cdot \gamma$ is a probability measure on $\mathbb{R}^d$ satisfying
{\rm(\ref{eq1})} with $b(x)=-x+v(x)$ , where $|v|\in L^1(\mu)$, one has
\begin{equation}\label{b1}
\int_{\mathbb{R}^d} f \bigl(\log(f+1)\bigr)^{\alpha}\, d\gamma \le
C(\alpha)\Bigl[
1 + \| v\|_{L^1(\mu)}
\Bigl(\log (1 + \| v\|_{L^1(\mu)})\Bigr)^{\alpha}\Bigr].
\end{equation}
\end{theorem}

It is possible that our bound
with $\alpha<1/4$ can be raised up to~$1/2$.

The natural infinite-dimensional  version of this result is considered in the next section.

The proof is based on two auxiliary results of independent interest.
Let $\{T_{t}\}_{t\ge0}$ denote the standard Ornstein--Uhlenbeck semigroup on $L^1(\gamma)$
defined by
$$
T_t\varphi(x)=\int_{\mathbb{R}^d} \varphi\Bigl(e^{-t}x-\sqrt{1-e^{-2t}}\, y\Bigr)\, \gamma(dy).
$$
Some elementary properties of this semigroup used below
can be found in \cite{B98}, \cite{B10}, and~\cite{B-umn18}.

Let   $\|\cdot\|_{K}$ denote the usual $1$-Kantorovich norm defined on bounded
signed measures $\sigma$ with $\sigma(\mathbb{R}^d)=0$ and finite first moment
by
$$
\|\sigma\|_K=\sup\biggl\{\int g\, d\sigma\colon g\in {\rm Lip}_1\biggr\},
$$
where ${\rm Lip}_1$ is the set of all $1$-Lipschitz functions on~$\mathbb{R}^d$.
It is readily seen that the supremum can be taken over the class of
$1$-Lipschitz smooth compactly supported functions.
This norm can be extended to the space of signed measures with finite first moment.
For example, we can set $\|\delta_0\|_K=1$ for Dirac's measure $\delta_0$ at zero
and then let $\|\sigma\|_K:=\|\sigma -\sigma(\mathbb{R}^d)\delta_0\|_K+|\sigma(\mathbb{R}^d)|$.
It is also possible to extend this norm by imposing the restriction~$g(0)=0$ when taking~$\sup$.

It is known that for every Borel probability
measure $\eta$ with finite first moment
on $\mathbb{R}^d$ there is a probability measure $\sigma$ on $\mathbb{R}^d\times\mathbb{R}^d$ with
projections $\eta$ and $\gamma$ such that it minimizes the
integral of $|x-y|$ over such measures and the corresponding
minimum is~$\|\eta -\gamma\|_K$. Such a measure $\sigma$ is called a $1$-optimal
transportation plan for the measures $\eta$ and~$\gamma$.
Hence the same is true for the pair of measures $c\eta$ and $c\gamma$ with any $c>0$:
their  $1$-optimal transportation plan is $c\sigma$ (its total mass is~$c$).
On this topic, see \cite{AG}, \cite{B18}, \cite{BK}, and~\cite{V}.

The next two lemmas are connected with properties of the
Ornstein--Uhlenbeck semigroup, but not with our equation.

\begin{lemma}\label{lem1}
Suppose that  $g\in L^1(\gamma)$ is a nonnegative function.
If the measure $g\cdot\gamma$ has a finite first moment, i.e.,
$$
K:=\bigl\|g\cdot\gamma - \|g\|_{L^1(\gamma)}\gamma\bigr\|_K
$$
 is finite, then
\begin{multline}\label{elog}
J_t(g)=\int_{\mathbb{R}^d} (T_{t}g) \bigl(\log (T_{t}g+1)\bigr)^{1/2} \,d\gamma
\\
\le
\|g\|_{L^1(\gamma)}\Bigl(\log (\|g\|_{L^1(\gamma)}+1)\Bigr)^{1/2} + 2^{-1} K t^{-1/2}
\quad \forall t \in [0, 1].
\end{multline}
In particular, if $g$ is a probability density
with respect to~$\gamma$, we have
$$
J_t(g)=\int_{\mathbb{R}^d} (T_{t}g) \bigl(\log (T_{t}g+1)\bigr)^{1/2} \,d\gamma
\le
\sqrt{\log{2}}+ 2^{-1} \|g\cdot\gamma -\gamma\|_K t^{-1/2}
\quad \forall t \in [0, 1].
$$
\end{lemma}
\begin{proof}
We employ Wang's log-Harnack inequality
for a nonnegative function $h\in L^1(\gamma)$
established in \cite{W97}, \cite{W1} (see also \cite{W13} and~\cite{W14})
in much greater generality:
\begin{equation}\label{ew1}
T_{t}\log h (x) \le \log T_{t}h(y) + \frac{1}{2}\frac{1}{e^{2t} - 1}|x - y|^2.
\end{equation}
Let us take $h= T_{t}g+1$ in Wang's inequality.
Then
$$
\bigl(T_{t}\log h (x) \bigr)^{1/2} \le
\bigl(\log T_{t}h(y) \bigr)^{1/2} + (4t)^{-1/2}|x - y|.
$$
Let $\sigma$ be a $1$-optimal transportation plan for $g\cdot\gamma$ and $\|g\|_{1}\gamma$,
where we write $\|g\|_{1}=\|g\|_{L^1(\gamma)}$ for simplicity in this proof.
Integrating the previous bound with respect to~$\sigma$
(we omit the indication of domain of integration below)
and observing that
\begin{align*}
\int \bigl(\log T_{t}h(y) \bigr)^{1/2}\, \sigma(dx\, dy)
&=
\|g\|_{1}\int \bigl(\log T_{t}h(y) \bigr)^{1/2}\, \gamma(dy)
\\
&\le \|g\|_1 \biggl(\log \int T_{t}h(y) \, \gamma(dy)\biggr)^{1/2}=
\|g\|_1 \bigl(\log (\|g\|_1+1)\bigr)^{1/2},
\end{align*}
where we have used that $T_th\ge 1$ and applied Jensen's inequality,
we arrive at the inequality
$$
\int g(x)\bigl(T_{t}\log h (x) \bigr)^{1/2}  \,\gamma(dx) \le
\|g\|_1 \bigl(\log (\|g\|_1+1)\bigr)^{1/2} + (4t)^{-1/2} \|g\cdot\gamma -I(g)\gamma\|_K ,
$$
which yields (\ref{elog}), since
$
T_{t}\bigl[\bigl(\log h\bigr)^{1/2}\bigr]
\le (T_{t}\log h )^{1/2}
$
and the integral of $gT_{t}\bigl[\bigl(\log h\bigr)^{1/2}\bigr]$ equals
the integral of $(T_{t}g) (\log h)^{1/2}$.
\end{proof}

It is worth noting that in the situation of this lemma the function
$T_tg$ belongs to the Gaussian Sobolev class~$W^{1,1}(\gamma)$, see
\cite[Proposition~3.5]{FH} or \cite[Proposition~5.12]{B-umn18}.

The Ornstein--Uhlenbeck operator $L$ is defined by
\begin{equation}\label{o-u}
L\varphi(x): = \Delta\varphi(x)  - \langle x, \nabla \varphi(x)\rangle
\end{equation}
for smooth functions $\varphi$.
It can be written as
$$
L\varphi = {\rm div}_{\gamma} \nabla \varphi,
$$
where for a smooth vector field $u$ we set
$$
{\rm div}_{\gamma} u(x)=
{\rm div}\, u(x)-\langle x,u(x)\rangle .
$$
For smooth compactly supported functions $\varphi$ and $\psi$
and a smooth vector field $u$ we have
$$
\int_{\mathbb{R}^d} \varphi L\psi\,d\gamma =
-\int_{\mathbb{R}^d} \langle \nabla \varphi,\nabla \psi\rangle \,d\gamma,
$$
$$
\int_{\mathbb{R}^d} \varphi \ {\rm div}_{\gamma} u \,d\gamma
=-\int_{\mathbb{R}^d} \langle \nabla \varphi,u\rangle \,d\gamma.
$$
These equalities extend to functions and vector fields from Gaussian Sobolev classes,
which is not used  below.

We shall now see that although for a general $\gamma$-integrable vector field~$u$
its $\gamma$-divergence can be a singular distribution, for every $s>0$, there is a  function
$T_s{\rm div}_{\gamma} u$ in~$L^1(\gamma)$ that satisfies the identity
\begin{equation}\label{iden}
\int_{\mathbb{R}^d} \varphi \ T_s{\rm div}_{\gamma} u \,d\gamma
=-\int_{\mathbb{R}^d} \langle \nabla T_s \varphi,u\rangle \,d\gamma
=-\int_{\mathbb{R}^d} e^{-s} \langle \nabla \varphi, T_s u\rangle \,d\gamma,
\quad \varphi\in C_0^\infty.
\end{equation}

\begin{lemma}\label{lem2}
Let $u$ be a Borel vector field on $\mathbb{R}^d$ such that
$|u|\in L^1(\gamma)$. Then, for each $s > 0$, there is a function
$$
T_{s}{\rm div}_{\gamma} u\in L^1(\gamma)
$$
satisfying {\rm(\ref{iden})} such that for the measure
$$
\nu_{s} := (T_{s}{\rm div}_{\gamma} u)\cdot\gamma
$$
we have
\begin{equation}\label{u1}
\|\nu_s\|=\|T_{s}{\rm div}_{\gamma} u\|_{L^1(\gamma)} \le
 \frac{e^{-s}}{\sqrt{1-e^{-2s}}} \| u\|_{L^1(\gamma)}
\le
\frac{1}{\sqrt{2s}} \| u\|_{L^1(\gamma)}.
\end{equation}
In addition, $\nu_s(\mathbb{R}^d)=0$ and
\begin{equation}\label{u2}
\|\nu_s\|_{K} \le e^{-s} \| u\|_{L^1(\gamma)}.
\end{equation}
This means that $T_s$ extends to the distributional $\gamma$-divergences of $\gamma$-integrable
vector fields as an operator with values in~$L^1(\gamma)$, i.e.,
$T_{s}{\rm div}_{\gamma}$ extends to a bounded operator from $L^1(\gamma,\mathbb{R}^d)$ to~$L^1(\gamma)$.
\end{lemma}
\begin{proof}
Suppose first that $u$ is smooth with compact support.
For all $\varphi \in C^{\infty}_{0}$ we have
$$
\int_{\mathbb{R}^d} T_{s}{\rm div}_{\gamma}u  \, \varphi \,d\gamma =
-\int_{\mathbb{R}^d} \langle u , \nabla T_{s}\varphi\rangle \,d\gamma,
$$
where
$$
\partial_h T_s\varphi(x)
=\frac{e^{-s}}{\sqrt{1-e^{-2s}}}
\int_{\mathbb{R}^d}
\varphi\Bigl(e^{-s}x-\sqrt{1-e^{-2s}}\, y\Bigr)\langle h,y\rangle\, \gamma(dy)
$$
for all $h\in\mathbb{R}^d$.
Hence
$$
|\nabla T_s\varphi(x)|\le \frac{e^{-s}}{\sqrt{1-e^{-2s}}}\|\varphi\|_\infty,
$$
which yields (\ref{u1}). Using that
$$
\nabla T_{s}\varphi =e^{-s}T_{s} \nabla \varphi,
$$
we obtain (\ref{u2}).
In the general case we take a sequence of smooth compactly supported vector
fields $u_j$ converging to $u$ in $L^1(\gamma)$. By (\ref{u1})
the sequence of smooth functions $T_s {\rm div}_\gamma  u_j$ converges
in~$L^1(\gamma)$. The limit is the desired function.
For $u_j$ equality (\ref{iden}) also holds for $\varphi\in C_b^\infty$,
in particular, for $\varphi=1$, hence the integral of $T_s{\rm div}_\gamma u_j$ against
$\gamma$ vanishes. Therefore,
$$
\int_{\mathbb{R}^d} T_s{\rm div}_\gamma u\, d\gamma=0.
$$

It is clear that (\ref{u2}) holds
and (\ref{iden}) remains true in the limit.
\end{proof}

Suppose now that a nonnegative function $f\in L^1(\gamma)$ satisfies the equation
$$
\Delta (f\cdot\gamma)-{\rm div}\, (fb\cdot\gamma)=0
$$
with
$$
b(x)=-x+v(x),
$$
where
$$
|v|\in L^1(\mu), \ \quad{i.e.,} \ f|v|\in L^1(\gamma).
$$
Using the Ornstein--Uhlenbeck operator defined by (\ref{o-u}),
we can rewrite the equation as
\begin{equation}\label{ell2}
\int_{\mathbb{R}^d} \bigl(
L\varphi + \langle v, \nabla \varphi\rangle \bigr) f\,d\gamma = 0
\quad \forall\, \varphi \in C^{\infty}_{0}.
\end{equation}
Let us set
$$
w := f \cdot v.
$$
By assumption $|w| \in L^{1}(\gamma)$.
We have in the sense of distributions
$$
Lf = {\rm div}_{\gamma} w,
$$
where
$$
{\rm div}_{\gamma} w=
{\rm div}\, w-\langle x,w\rangle .
$$
For smooth $w$ and $f$ this  would be
$$
\int_{\mathbb{R}^d} \varphi Lf\,d\gamma =
\int_{\mathbb{R}^d} \varphi \ {\rm div}_{\gamma} w \,d\gamma,
\quad \varphi \in C^{\infty}_{0}.
$$
It is readily seen that identity (\ref{ell2}) remains valid for all $\varphi$
from the class $\mathcal{S}(\mathbb{R}^d)$ of smooth rapidly decreasing functions.
Since $T_s\varphi\in \mathcal{S}(\mathbb{R}^d)$ for all $\varphi\in C_0^\infty$ and $s\ge 0$,
we obtain the identity
\begin{equation}\label{ell3}
\int_{\mathbb{R}^d} \bigl(
LT_s\varphi + \langle v, \nabla T_s\varphi\rangle \bigr) f\,d\gamma = 0
\quad \forall\, \varphi \in C^{\infty}_{0}, \ s\ge 0.
\end{equation}
Therefore, in $L^1(\gamma)$ we have
\begin{equation}\label{Tf}
T_{t}f - f= \int_{0}^t T_{s} {\rm div}_{\gamma} w\,ds,
\end{equation}
where the last integral exists in $L^1(\gamma)$ by Lemma~\ref{lem2}.
Indeed, the integrals with respect to $\gamma$ of both sides multiplied
by any $\varphi\in C_0^\infty$ coincide, because  by (\ref{iden}) and (\ref{ell3})
we have
$$
\int_{\mathbb{R}^d} \varphi T_{s}
{\rm div}_{\gamma} w\,d\gamma=-\int_{\mathbb{R}^d} \langle \nabla T_s\varphi,w\rangle \,
d\gamma=\int_{\mathbb{R}^d} LT_s\varphi f\,d\gamma,
$$
but the integral of the right-hand side in $s$ over $[0,t]$ equals the integral
of $f(T_t\varphi -\varphi)$ against~$\gamma$, which is the integral of $\varphi(T_tf-f)$
by the symmetry of~$T_t$.

\begin{proposition}\label{p1}
Under the assumptions of Theorem~\ref{mt1}  we have
$$
\|f\cdot\gamma  - \gamma\|_{K} \le  \|fv\|_{L^1(\gamma)}=
\| v\|_{L^1(\mu)}.
$$
In addition,
$$
\|T_tf-f\|_{L^1(\gamma)}\le (2t)^{1/2} \| v\|_{L^1(\mu)}.
$$
\end{proposition}
\begin{proof}
It follows by (\ref{Tf}) and (\ref{u2})
that for all $t>0$ and $\varphi\in C_0^\infty(\mathbb{R}^d)$ we have
$$
\biggl|\int_{\mathbb{R}^d} \varphi T_tf\, d\gamma
- \int_{\mathbb{R}^d} \varphi f\, d\gamma\biggr|
=\biggl| \int_{\mathbb{R}^d}
\int_{0}^t \varphi T_{s} {\rm div}_{\gamma} w\,ds \, d\gamma
\biggr|
\le \|fv\|_{L^1(\gamma)} \|\nabla\varphi\|_\infty
\int_0^t e^{-s}\, ds.
$$
It remains to recall that  $T_{t}f$ converges to $1$ in $L^1(\gamma)$
as $t\to\infty$. The second estimate follows similarly
by using~(\ref{u1}). The passage to all $1$-Lipschitz functions
in place of smooth compactly supported ones is easily justified by Fatou's theorem
taking into account that the Gaussian measure has finite first moment.
\end{proof}

The bound $\|T_tf-f\|_{L^1(\gamma)}\le Ct^{1/2}$ can be
regarded as the inclusion  of $f$ in a certain fractional Besov type space
with respect to~$\gamma$ in the spirit of \cite{BKP} and~\cite{BKZ}.

It is worth noting that this proposition yields that
the norm $|x|$ is $\mu$-integrable, so  $\|\mu\|_K<\infty$, which is not obvious
in advance (and has not been assumed in the proof)
and does not follow from the $\gamma$-integrability of the function
$f(\log (f+1))^\alpha$ with~$\alpha<1/4$,
unlike the case where $f(\log (f+1))^{1/2}$ is $\gamma$-integrable
(which holds if $f\in W^{1,1}(\gamma)$ or at least $f\in BV(\gamma)$, but this
can fail in our situation according to~\cite{BPS}). It is also known
(see \cite[Lemma~2.3]{AF}) that
the bound $\|T_tf-f\|_{L^1(\gamma)}\le Ct^{1/2}$ holds if $f\in BV(\gamma)$.

\begin{proof}[Proof of Theorem~\ref{mt1}]
Let $t_n := n^{-\beta}$, where $\beta>1$ will be picked later.
Set
$$
g_{n} := |T_{t_{n + 1}}f - T_{t_{n}}f|,
$$
$$
V:=\| v\|_{L^1(\mu)}.
$$
For simplicity, we omit indication of $\mathbb{R}^d$ when integrating over the whole space.
Let us observe that by Proposition~\ref{p1} there is a number $C(\beta)$, depending
only on~$\beta$, such that
\begin{equation}\label{fe1}
\int g_{n}\,d\gamma \le C(\beta) n^{-(1+\beta)/2}V,
\quad n\ge 1.
\end{equation}
In addition,
\begin{equation}\label{fe2}
\int g_{n}\bigl(\log(g_{n}+1)\bigr)^{1/2}\,d\gamma
\le 2\sqrt{\log 2}+ n^{\beta/2}V.
\end{equation}
This bound is obtained as follows:
$$
\int g_{n}\bigl(\log(g_{n}+1)\bigr)^{1/2}\,d\gamma
\le
\int (T_{t_{n + 1}}f)\bigl(\log(T_{t_{n + 1}}f+1)\bigr)^{1/2}\,d\gamma
+
\int (T_{t_{n}}f)\bigl(\log(T_{t_{n}}f+1)\bigr)^{1/2}\,d\gamma,
$$
because $|f_1-f_2|\bigl(\log (|f_1-f_2|+1)\bigr)^{1/2}$
is dominated pointwise by the sum of the functions
$f_{1}\bigl(\log (f_1+1)\bigr)^{1/2}$ and
$f_{2}\bigl(\log (f_2+1)\bigr)^{1/2}$
whenever $f_1,f_2\ge 0$. It remains to apply
Lemma~\ref{elog} and Proposition~\ref{p1}.

By H\"older's inequality and (\ref{fe1}), (\ref{fe2}) we have
\begin{align*}
\int g_{n}\bigl(\log(g_{n}+1)\bigr)^{\alpha}\,d\gamma
&\le \biggl[\int g_{n}\,d\gamma \biggr]^{1-2\alpha}
\biggl[\int g_{n}\bigl(\log(g_{n}+1)\bigr)^{1/2} \,d\gamma
\biggr]^{2\alpha}
\\
&\le C(\alpha,\beta) n^{\alpha\beta -(1+\beta-2\alpha-2\alpha\beta)/2}V
\\
&+C(\alpha,\beta)n^{-(1+\beta)(1-2\alpha)/2} V^{1-2\alpha}.
\end{align*}
If $\alpha<1/4$ is fixed, we can find $\beta$ large enough
(namely, $\beta> (2\alpha +1)/(1-4\alpha)$)
so that the powers obtained will be less than $-1$.
Since
$$
f \le T_{1}f + \sum_{n}g_{n},
$$
for obtaining
the desired inequality it remains to apply
the triangle inequality for the corresponding Orlicz norm. To this end,
we estimate the Luxemburg norms of~$g_n$. Recall (see~\cite{KR}) that the Luxemburg
norm is equivalent to the Orlicz norm and in our case is defined by
$$
\|g\|_L=\inf
\biggl\{s>0\colon \int \frac{g}{s}
\Bigl(\log \Bigl(\frac{g}{s}+1\Bigr)\Bigr)^{\alpha}\, d\gamma\le 1\biggr\}.
$$
Now let us bound
$$
\int \frac{g_n}{s}
\Bigl(\log \Bigl(\frac{g_n}{s}+1\Bigr)\Bigr)^{\alpha}\, d\gamma
$$
via $\|g_n\|_{L^1(\gamma)}$ and the integral of
$g_{n}\bigl(\log(g_{n}+1)\bigr)^{\alpha}$:
\begin{align*}
\int \frac{g_n}{s}
\Bigl(\log \Bigl(\frac{g_n}{s}+1\Bigr)\Bigr)^{\alpha}\, d\gamma
&=
\int \frac{g_n}{s}
\Bigl(\log \Bigl(\frac{g_n + s}{s}\Bigr)\Bigr)^{\alpha}\, d\gamma
\\
&\le
\int \frac{g_n}{s}
\Bigl(\log (g_n + 1) + \log (1/s + 1)\Bigr)^{\alpha}\, d\gamma
\\
&\le
\int \frac{g_n}{s}
\Bigl[\Bigl(\log (g_{n}+1)\Bigr)^\alpha +
\Bigl(\log (1/s + 1)\Bigr)^{\alpha}\Bigr]\, d\gamma
\\
&\le
\frac{1}{s} \int g_{n}\Bigl(\log (g_n+1)\Bigr)^{\alpha} +
\frac{1}{s}\Bigl(\log (1/s + 1)\Bigr)^{\alpha}
\int g_n\, d\gamma.
\end{align*}
Hence
$$
\|g_n\|_L\le C(\alpha) n^{-  \delta }
(1 + V) \Bigl(1 + \bigl(\log(1 + V)\bigr)^{\alpha}\Bigr),
\ \delta = \delta(\alpha) > 1.
$$
Finally, we obtain convergence of the Luxemburg norms of the functions~$g_n$ and, consequently,
$$
\|f\|_{L} \le C(\alpha)(1 + V)
\Bigl(1 + \bigl(\log (1 + V)\bigr)^{\alpha}\Bigr).
$$
It remains to bound the integral of $f(\log (f+1))^\alpha$.
Let $L : = \|f\|_{L}$, $g: = f/L$.
Using the same arguments as above we have
\begin{align*}
\int f\bigl(\log (f+1)\bigr)^\alpha\, d\gamma
&\le
L \int g\bigl(\log (g+1)\bigr)^\alpha\, d\gamma +
L (\log(L + 1))^{\alpha} \int_{\mathbb{R}^d}g\,d\gamma
 \\
&\le L + (\log(L + 1))^{\alpha}
\\
&\le C(\alpha)(1 + V)\Bigl(1 + \bigl(\log(1 + V)\bigr)^{\alpha}\Bigr).
\end{align*}
Thus, we have obtained the desired bound.
\end{proof}

\begin{remark}
\rm
As a corollary, we can obtain the following bound on the tail distribution of~$f$:
$$
\gamma \bigl(x\colon f(x) > \lambda \bigr) \le
C\frac{\log \log \lambda}{\lambda \log^{1/3}\lambda}, \quad \lambda \ge \lambda_0.
$$
Indeed, for the operator $A_1$ defined by
$$
A_{t} := \frac{1}{t}\int_0^t T_{s}\,ds,
$$
Talagrand's result \cite{T} yields the bound
$$
\gamma
\bigl(x\colon  A_{1}f(x) > \lambda \bigr)
\le \frac{C \log\log \lambda}{\lambda\log \lambda}, \ \lambda \ge \lambda_0.
$$
Then
\begin{align*}
\gamma \bigl(x\colon f(x) > \lambda \bigr)
&\le
\gamma \bigl(x\colon A_{t}f(x) > \lambda/2 \bigr) +
\gamma \bigl(x\colon |f - A_{t}f|(x) > \lambda/2 \bigr)
\\
&
\le
\frac{1}{\lambda t}\frac{C' \log (\log \lambda + \log t)}{\log \lambda + \log t} +
t^{1/2}\frac{2}{\lambda}\| v\|_{L^1(\mu)}.
\end{align*}
Taking $t := (\log\lambda)^{-2/3}$ and assuming
that $\lambda$ is sufficiently large, we obtain the announced bound.
Let us observe that a somewhat worse bound can be obtained from Lehec's result~\cite{Leh}
for $T_t$ in place of~$A_t$.
\end{remark}

\begin{remark}
\rm
The following inequality  was established in \cite{BWS15}
for functions $f\in W^{1,1}(\gamma)$:
$$
\|f - 1\|_{L^1(\gamma)}^{2} \le
2\| f \cdot \gamma  - \gamma \|_{K} \| \nabla f\|_{L^1(\gamma)}.
$$
This inequality is a generalization of the classical Hardy--Landau--Littlewood
inequality for functions on the real line (see also \cite{BS15} and~\cite{BWS16}).
Using the same reasoning as in \cite{BWS15} and the estimates obtained above
one can prove  that for $f$ in Theorem~\ref{mt1} one has
$$
\|f - 1\|_{L^1(\gamma)}^{2} \le 2
\| f \cdot \gamma  - \gamma \|_{K} \| v\|_{L^1(\mu)}
\le 2 \| v\|_{L^1(\mu)}^2.
$$
A different derivation of this bound has been given by A.F.~Miftakhov
(see~\cite{BMS}).
\end{remark}

\begin{remark}
\rm
It is plain that the main theorem is based on two ingredients:  an a priori bound
on the integral of $T_tf[\log (T_tf+1)]^{1/2}$ by $Ct^{-1/2}$ and the estimate
$\|T_tf-f\|_1\le Ct^{1/2}$. No special properties of the semigroup
are needed to derive the final assertion. In turn, the first bound
is a corollary of Wang's log-Harnack inequality and a finite Kantorovich distance
between $\mu$ and~$\gamma$.
Wang's method applies to a broad class of diffusion semigroups
(see \cite{ATW}, \cite{RW}, \cite{W97}, \cite{W1}, \cite{W13} and~\cite{W14}).
Therefore, within the class of such semigroups, the remaining questions are about
the Kantorovich distance and the estimate for $\|T_tf-f\|_1$,
which are handled in Proposition~\ref{p1}.
In turn, these estimates are based on the following properties of the semigroup:
$$
|\nabla T_t\varphi|\le e^{-Ct}\|\nabla\varphi\|_\infty,
\quad
|\nabla T_t\varphi|\le Ct^{-1/2}\|\varphi\|_\infty.
$$
Both properties hold for a broad class of diffusion semigroups (see below).
Of course, constants may be different and the Kantorovich norm
must be taken with respect to the intrinsic metric. For example,
in place of the standard Ornstein--Uhlenbeck semigroup we can consider the semigroup
$$
T_t^B\varphi (x)=\int \varphi\Bigl(e^{-tB}x-\sqrt{1-e^{-2tB}}\, y\Bigr)\, \gamma_B(dy),
$$
where $B$ is a positive definite operator on $\mathbb{R}^d$ and $\gamma_B$ is the centered
Gaussian measure with covariance $B^{-1}$. The measure $\gamma_B$ is invariant for~$\{T_t^B\}_{t\ge0}$
and satisfies the stationary equation with the drift $-Bx$. Assume that
$B\ge \beta_1 I$, where $\beta_1>0$ is the minimal eigenvalue of~$B$.
Suppose that $\mu$ satisfies the stationary equation with $b(x)=-Bx+v(x)$.
Then $\mu$ is absolutely continuous with respect to~$\gamma_B$ and
the integral of $f[\log (f+1)]^\alpha$ for $f=d\mu/d\gamma_B$
is finite for all~$\alpha<1/4$. Moreover, if $\beta_1\ge 1$, then
(\ref{b1}) remains valid, in the general case a constant depending
on $\beta_1$ will appear.

Indeed,  we have
$$
\nabla T_t^B\varphi =e^{-tB}T_t^B\nabla\varphi.
$$
Therefore, in the corresponding analog of Proposition~\ref{p1}
for any $1$-Lipschitz function $\varphi$ we have
$$
\int_0^\infty\int_{\mathbb{R}^d}
\langle e^{-tB}T_t^B \nabla\varphi ,u\rangle\, d\gamma_B\, dt
\le
\int_{\mathbb{R}^d} \int_0^\infty
e^{-t\beta _1}|u|\, dt\, d\gamma_B=\beta_1^{-1}\|u\|_{L^1(\gamma_B)}.
$$
Hence in the bound for $\|f\cdot\gamma_B- \gamma_B\|_K$ we have to replace
$|v|$ by $\beta_1^{-1}|v|$. Next, for estimating $\|T_t^Bf-f\|_{L^1(\gamma_B)}$
we use the equality
$$
\partial_h T_s^B\varphi(x)
=
\int_{\mathbb{R}^d}
\varphi\Bigl(e^{-sB}x-\sqrt{1-e^{-2sB}}\, y\Bigr)
\langle e^{-sB}(1-e^{-2sB})^{-1/2} h, By\rangle\, \gamma_B(dy).
$$
The integral of $y\mapsto |\langle z,B^{1/2}y\rangle|$ with respect to $\gamma_B$
is estimated by~$|z|$,
and
$$
|e^{-sB}(1-e^{-2sB})^{-1/2}B^{1/2}h|\le (2s)^{-1/2}|h|.
$$
Hence $|\nabla T_s^B\varphi|\le (2s)^{-1/2}\|\varphi\|_\infty$, which gives
the same estimate for $\|T_t^Bf-f\|_{L^1(\gamma_B)}$ as in Proposition~\ref{p1}.
Wang's log-Harnack inequality also holds in this case. If $\beta_1\ge 1$, then
it holds without any change; for any $\beta_1>0$ it holds with the additional
factor $e$ in front of~$|x-y|^2$.
This follows from the case $B=I$ by changing variables and observing that the norm of
$e^{t}(1-e^{-2t})^{1/2}B^{1/2} (I-e^{-2tB})^{-1/2}$ is estimated by~$e$ if $\beta_1>0$
and is estimated by~$1$ if $\beta_1\ge 1$.
Thus, the proof given above applies.
\end{remark}

Let us formulate an analog of Theorem~\ref{mt1} for manifolds.
Let  $(M,\varrho)$ be a connected complete Riemann manifold. In place of the
Ornstein--Uhlenbeck operator we consider the operator
$$
L\varphi=\Delta  \varphi+\langle Z,\nabla\varphi\rangle,
$$
where $Z$ is a smooth
vector field  satisfying the curvature condition
$$
{\rm Ric}(X,X)-\langle \nabla_X, Z\rangle\ge
-K |X|^2
$$
with some number $K\in\mathbb{R}$.
In the case of the standard Ornstein--Uhlenbeck operator we have
${\rm Ric}=0$, $K=-1$, $Z(x)=-x$, $\langle \nabla_X, Z\rangle=-|X|^2$.
It is known  (see \cite[Theorem~2.3.3]{W14})
that the diffusion semigroup $\{P_t\}$
generated by the operator $L$
satisfies the inequality
$$
P_t(\log f)(x)\le \log P_tf(y)+\frac{K}{2(1-e^{-2tK})} \varrho(x,y)^2
$$
for all $P_t$-integrable functions $f>0$. It is also known
(see, in particular, \cite[Theorem~2.3.1]{W14}) that
$$
|\nabla P_t \varphi|\le e^{tK}P_t|\nabla \varphi|,
\quad
|\nabla P_t \varphi|\le Ct^{-1/2}\|\varphi\|_\infty.
$$
In addition, there is a probability measure $\mu_0$ on~$M$ satisfying the equation
$L^*\mu_0=0$ (see \cite[Theorem~3.4]{BRW01}).
It follows from the previous remark that the following assertion is true.

\begin{theorem}
Suppose that a probability measure $\mu$ on~$M$ satisfies
the perturbed  equation
$$L_v^*\mu=0,
$$
 where
 $$
 L_v\varphi=L\varphi+\langle v,\nabla \varphi\rangle
 $$
and $v$ is a Borel vector field on~$M$ such that $|v|\in L^1(\mu)$. Then
$\mu$ has a density $f$ with respect to~$\mu_0$ and $f(\log (f+1))^\alpha\in L^1(\mu_0)$
for all $\alpha<1/4$. Moreover, there is a number $C$ depending on $\alpha$ and $K$ such that
$$
\int_{M} f \bigl(\log(f+1)\bigr)^{\alpha}\, d\mu_0 \le
C(\alpha,K)\Bigl[
1 + \| v\|_{L^1(\mu)}
\Bigl(\log (1 + \| v\|_{L^1(\mu)})\Bigr)^{\alpha}\Bigr].
$$
\end{theorem}

In the case of $\mathbb{R}^d$ this result applies to smooth $Z$ such that
$$
\langle Z(x)-Z(y), x-y\rangle \le - k |x-y|^2
$$
for some number $k>0$. In particular, one can take for $Z$ a negative definite linear
operator.

\begin{remark}
\rm
Since our main assumption is the integrability of $v$ with respect to the solution~$\mu$
that typically is not explicitly given,
it is of interest to ensure this integrability in terms of $v$ (or~$b(x)=-x+v(x)$) without using~$\mu$.
A sufficient condition can be given by using Lyapunov functions: if $v$ is locally
bounded, it suffices to have a
function $V\in C^2(\mathbb{R}^d)$ such that $V(x)\to+\infty$ as $|x|\to +\infty$
and
$$
\Delta V(x)+\langle b(x), \nabla V(x)\rangle \le C-|v(x)|
$$
with some positive constant $C$. In this case $\| v\|_{L^1(\mu)}\le C$.
For example, if
$$\langle v(x),x\rangle \le -C_1<-d \quad \hbox{outside some ball and} \quad
|v(x)|\le C_2\exp(|x|^2/2),
$$
then one can take $V(x)=\exp(|x|^2/2)$ and conclude that there is a unique probability solution
$\mu$ and $|v|\in L^1(\mu)$.
\end{remark}

\section{Infinite-dimensional extensions}

The results of the previous section admit straightforward infinite-dimensional extensions.
Let $X$ be a locally convex space, let $X^*$ be the topological dual of~$X$,
 and let $\gamma$ be a centered Radon Gaussian measure on~$X$.
This means that $\gamma$ is a Borel probability measure such that for every Borel set $B$ and
every $\varepsilon>0$ there is a compact set $K\subset B$ with $\gamma(B\backslash K)<\varepsilon$,
and, in addition, every functional $l\in X^{*}$ is a centered Gaussian random variable, i.e.,
is either zero almost  everywhere or
$$
\gamma (x\colon l(x)<s)=\frac{1}{\sqrt{2\pi \sigma}}\int_{-\infty}^s
e^{-u^2/(2\sigma)}\, du,
$$
where $\sigma=\|l\|_{L^2(\gamma)}^2$.

The Cameron--Martin space $H$ of $\gamma$ consists of all vectors with finite norm
$$
|h|_H=\sup \{l(h)\colon l\in X^{*}, \, \|l\|_{L^2(\gamma)}\le 1\}.
$$
This is a separable Hilbert space with respect to the norm $|\,\cdot\,|_H$
(called the Cameron--Martin norm), the corresponding inner product
is denoted by $(h,k)_H$.

The most important example is the countable power of the standard Gaussian
measure on the real line, which is defined on the space $\mathbb{R}^\infty$ of all
real sequences (or on a suitable Hilbert subspace of full measure). The corresponding
Cameron--Martin space is the usual space~$l^2$. Moreover, by the celebrated Tsirelson theorem,
every centered Radon Gaussian measure with an infinite-dimensional
Cameron--Martin space is isomorphic to this particular example by means of a measurable
linear mapping. Hence we can assume without loss of generality that $\gamma$ below is this
countable product on~$\mathbb{R}^\infty$. Given a Borel vector field
$$
v\colon X\to H,
$$
one can define solutions to the stationary equation
$$
L_b^{*}\mu=0, \quad b(x)=-x+v(x),
$$
as follows.
First, dealing with~$X=\mathbb{R}^\infty$,
  we introduce the class $\mathcal{F}\mathcal{C}_0$ of test functions of the form
  $$
\varphi(x)=\varphi_0(x_1,\ldots,x_n), \quad \varphi_0\in C_0^\infty(\mathbb{R}^n).
  $$
The reader is warned that this class is not a linear space.
For a general locally convex space, an analogous class consists of cylindrical functions.

Next, we define the
Ornstein--Uhlenbeck operator $L$ on $\mathcal{F}\mathcal{C}_0$ by
$$
L\varphi (x)=\sum_{i=1}^\infty [\partial_{x_i}^2\varphi(x) -x_i \partial_{x_i}\varphi(x)].
$$
Obviously, for each $\varphi\in \mathcal{F}\mathcal{C}_0$ this is a finite sum.
Let
$$
v=(v_i), \quad |v(x)|_H^2=\sum_{i=1}^\infty v_i(x)^2.
$$
The operator $L_b$ with $b(x)=-x+v(x)$ is defined by
$$
L_b\varphi (x)=L\varphi(x)+ \sum_{i=1}^\infty v_i(x) \partial_{x_i}\varphi(x).
$$
Finally, if $\mu$ is a Borel probability measure on~$X$ such that
$v_i\in L^1(\mu)$ for all $i$ and
$$
\int_X L_b\varphi\, d\mu=0
\quad \forall\, \varphi\in \mathcal{F}\mathcal{C}_0,
$$
then we say that $\mu$ satisfies the equation $L_b^{*}\mu=0$.
Note that $L_b\varphi$ is bounded, since if $\varphi$ depends on $x_1,\ldots,x_n$, then
the functions $x_i\partial_{x_i}\varphi$ are bounded.

In the case of an abstract locally convex space~$X$ the definition is analogous:
there are even two similar options. For a class
of test functions one can use functions of the form $\varphi(l_1,\ldots,l_n)$,
where $\varphi\in C_0^\infty(\mathbb{R}^n)$, $l_i\in X^{*}$. Alternatively, one can fix
a sequence $\{l_i\}\subset X^{*}$ (say, if there is a sequence separating points) and take only $l_i$
from this sequence. However, the definition of $L_b$ becomes a bit more technical
(see \cite[Section~4]{B-umn18}), because
it involves also~$H$. Let $\{e_i\}$ be an orthonormal basis in~$H$ such that
there is a sequence $\{\widehat{e}_j\}\subset X^*$ for which $\widehat{e}_i(e_j)=\delta_{ij}$.
Every functional $l\in X^{*}$ has a continuous restriction to~$H$, hence there is a vector
$\widetilde{l}\in H$ with $l(u)=(\widetilde{l},u)_H$ for all~$u\in H$.
Note that $\widetilde{\widehat{h}}=h$.
The operator $L_b\varphi$ is defined by
$$
L_b\varphi=L\varphi+\sum_{i=1}^\infty (v,e_i)_H \partial_{e_i}\varphi,
$$
$$
L\varphi=\sum_{i=1}^\infty [\partial_{e_i}^2\varphi-\widehat{e}_i\partial_{e_i}\varphi],
$$
which for cylindrical functions as above can be written as
$$
L\varphi=\sum_{j,k\le n} (\widetilde{l}_j,\widetilde{l}_k)_H \partial_{x_j}
\partial_{x_k}\varphi(l_1,\ldots,l_n)
-\sum_{j=1}^n l_j \partial_{x_j}\varphi(l_1,\ldots,l_n).
$$
If we use $l_j=\widehat{e_j}$, then we arrive at the simple expression used above
in the case of~$\mathbb{R}^\infty$.

It is readily seen that if $v_i^n$ is the conditional expectation
of $v_i$ with respect to the measure $\mu$ and the $\sigma$-field $\mathcal{B}_n$
generated by $x_1,\ldots,x_n$, then the projection $\mu_n$ of $\mu$ to $\mathbb{R}^n$
satisfies the finite-dimensional equation
$$
L_{b^n}^*\mu_n=0
$$
with $b^n(x)=-x+v^n(x)$ on $\mathbb{R}^n$,
$v^n=(v_1^n,\ldots,v_n^n)$. Actually, our infinite-dimensional equation is equivalent
to this system of finite-dimensional equations for projections.

\begin{theorem}\label{t2}
Let $\mu$ be  a Borel probability measure on~$X$ such that $|v|_H\in L^1(\mu)$
and $L_b^*\mu=0$. Then $\mu$ is absolutely continuous with respect to $\gamma$
and for $f:=d\mu/d\gamma$ we have
\begin{equation}\label{b2}
\int_X f \bigl(\log(f+1)\bigr)^{\alpha}\, d\gamma \le
C(\alpha)\Bigl[
1 + \bigl\| |v|_H \bigr\|_{L^1(\mu)}
\Bigl(\log\bigl(1 + \bigl\| |v|_H \bigr\|_{L^1(\mu)}\bigr)\Bigr)^{\alpha}\Bigr],
\end{equation}
where $\alpha<1/4$ and $C(\alpha)$ are the same numbers as in Theorem~{\rm\ref{mt1}}.
\end{theorem}
\begin{proof}
It is readily seen that the finite-dimensional densities $f_n=d\mu_n/d\gamma_n$
regarded as functions in $L^1(\gamma)$ form a martingale with respect to $\gamma$
and the $\sigma$-fields~$\mathcal{B}_n$.
By the property of conditional expectations we have
$$
\| v^n\|_{L^1(\mu_n)}\le \bigl\| |v|_H\bigr\|_{L^1(\mu)}
$$
By the main theorem this martingale is
uniformly integrable, hence converges in the weak topology of $L^1(\gamma)$ and almost everywhere
to some
function $f\in L^1(\gamma)$. Then, for every
$\varphi\in \mathcal{F}\mathcal{C}_0$, the integral of $\varphi f$ with respect to~$\gamma$
equals the integral of $\varphi$ with respect to~$\mu$. Hence $\mu=f\cdot\gamma$.
Estimate~(\ref{b2}) follows by Fatou's theorem.
\end{proof}

Unlike the case of $\mathbb{R}^d$, the absolute continuity of $\mu$ with respect to $\gamma$
is also a substantial novelty of this theorem.

In the infinite-dimensional case it is important to distinguish between
the Cameron--Martin space norm of $v$ and a weaker norm that arises if we consider
$\mu$ and $\gamma$ on some continuously embedded Hilbert space $E$ of full measure.
The integrability of $\|v\|_E$ does not guarantee the absolute continuity of~$\gamma$
even if $v$ still takes values in the Cameron--Martin space~$H$. This is why
in the infinite-dimensional case we avoid writing
$\|v\|_{L^1(\mu)}$ in place of $\bigl\| |v|_H \bigr\|_{L^1(\mu)}$, as we did in~$\mathbb{R}^d$.

\begin{remark}
\rm
The following analog of Lemma~\ref{lem2} holds.
Let $u\colon X\to H$ be a Borel vector field such that
$|u|_H\in L^1(\gamma)$. Then, for each $s > 0$, there is a
function
$$
T_{s}{\rm div}_{\gamma} u\in L^1(\gamma)
$$
 satisfying the identity
$$
\int_X \varphi T_{s}{\rm div}_{\gamma} u\, d\gamma
=-\int_X ( D_H T_s \varphi, u)_H \, d\gamma
= -\int_X e^{-s}( D_H \varphi, T_s  u)_H \, d\gamma, \quad
\varphi\in \mathcal{F}\mathcal{C}_0
$$
and the bound
\begin{equation}\label{u3}
\|T_{s}{\rm div}_{\gamma} u\|_{L^1(\gamma)} \le
 \frac{e^{-s}}{\sqrt{1-e^{-2s}}}\bigl\| |u|_H\bigr\|_{L^1(\gamma)}.
 \end{equation}
 This means that $T_{s}{\rm div}_{\gamma}$ extends to a bounded operator from
 $L^1(\gamma,H)$ to~$L^1(\gamma)$.
 In addition, in the situation of Theorem~\ref{t2} for $w=fv$ equality  
 (\ref{Tf}) holds. 
 
For the proof we take mappings $u^j\colon X\to \mathbb{R}^d$ with components
of class $\mathcal{F}\mathcal{C}_0$ such that $|u-u^j|_H\to 0$ in $L^1(\gamma)$.
It follows from Lemma~\ref{lem2} that the functions $T_{s}{\rm div}_{\gamma} u^j$
converge in~$L^1(\gamma)$. The limit will be denoted by $T_{s}{\rm div}_{\gamma} u$.
Obviously, it satisfies the desired identity and inequality~(\ref{u3}).
Equality (\ref{Tf}) for $w=fv$ follows from the finite-dimensional case applied 
to $w^n=f_nv^n$, because by the martingale convergence theorem 
we have convergence of $|v-v^n|_H$ to zero in $L^1(\gamma)$, which along with 
(\ref{u3}) enables us to pass to the limit in (\ref{Tf}) for $w^n$. 

It follows from the aforementioned identity that if $|u|_H\in L^p(\gamma)$ with some $p>1$, then
$$
T_{s}{\rm div}_{\gamma} u\in L^p(\gamma),
$$
 since for $\varphi\in L^{p/(p-1)}(\gamma)$
one has $|D_H T_s \varphi|_H\in L^{p/(p-1)}(\gamma)$ and
$\bigl\| |D_H T_s \varphi|_H\bigr\|_{L^{p/(p-1)}(\gamma)}$
is estimated through~$\|\varphi\|_{L^{p/(p-1)}(\gamma)}$. Moreover, the order of integrability
of $T_{s}{\rm div}_{\gamma} u$ can be increased by writing $T_s=T_{s-\delta}T_\delta$ and using
that by the hypercontractivity (see, e.g.,~\cite{B98})
 $T_{s-\delta}$ takes $L^p(\gamma)$ to $L^{q}(\gamma)$ with
$q=e^{2s-2\delta}(p-1)+1$.
Note that this does not help much for estimating $f$, because in our main situation
$u=fv$, so that even if $v$ is bounded, some a priori information is needed
about the integrability of~$f$ (and in the general case $f$ can fail
to be integrable to a power larger than~$1$, so increasing integrability is only possible
in a logarithmic scale).
 \end{remark}

In addition, Proposition~\ref{p1} also extends to infinite dimensions with the following modification:
in the definition of the Kantorovich norm, one should take the supremum over the intersection
of $\mathcal{F}\mathcal{C}_0$ with the class ${\rm Lip}_1(H)$ of
Borel functions that are $1$-Lipschitz along the Cameron--Martin space
or over the whole class  ${\rm Lip}_1(H)$. By definition the class ${\rm Lip}_1(H)$
consists of all Borel functions $\varphi$ for which
$$
|\varphi(x+h)-\varphi(x)|\le C|h|_H, \quad x\in X,\, h\in H.
$$

The corresponding
definition is this. For Borel probability measures $\mu$ and $\nu$ integrating
all Borel functions that are Lipschitz along
the Cameron--Martin space we set
$$
\|\mu-\nu\|_{K,H}=\sup\biggl\{
\int_X \varphi\, d\mu - \int_X \varphi\, d\nu\colon \varphi\in {\rm Lip}_1(H)\biggr\},
$$
where ${\rm Lip}_1(H)$

\begin{proposition}
Under the assumptions of Theorem~{\rm\ref{t2}}  we have
$$
\|f\cdot\gamma  - \gamma\|_{K,H} \le  \bigl\|f|v|_H\bigr\|_{L^1(\gamma)}=
\bigl\| |v|_H\bigr\|_{L^1(\mu)}.
$$
In addition,
$$
\|T_tf-f\|_{L^1(\gamma)}\le (2t)^{1/2} \bigl\| |v|_H\bigr\|_{L^1(\mu)}.
$$
\end{proposition}
\begin{proof}
The second bound follows immediately from the finite-dimensional case
as in the proof of Theorem~\ref{t2}. To obtain the first bound we need to pass from
the  $\mathcal{F}\mathcal{C}_0\cap {\rm Lip}_1(H)$  to ${\rm Lip}_1(H)$
in the inequality
$$
\biggl|\int_X \varphi f\, d\gamma - \int_X \varphi\, d\gamma\biggr|
\le \bigl\| |v|_H\bigr\|_{L^1(\mu)}.
$$
We first observe that every function in ${\rm Lip}_1(H)$ is $\gamma$-integrable
(see \cite[Theorem~4.5.7]{B98}). Hence by Fatou's theorem we conclude that it is also $\mu$-integrable.
Now applying Fatou's theorem once again we conclude that the previous inequality extends to~${\rm Lip}_1(H)$.
\end{proof}

We emphasize that in this proposition we have shown that
functions from ${\rm Lip}_1(H)$ are $\mu$-integrable, which is not obvious in advance.
This property enables us to extend the class~$\mathcal{F}\mathcal{C}_0$, with respect
to which the equation is defined, to the larger class $\mathcal{F}\mathcal{C}_b$, in which
representing functions $\varphi_0$ are taken in the class $C_b^\infty(\mathbb{R}^n)$. The advantage
of $\mathcal{F}\mathcal{C}_b$ is that it is a linear space. However,
the problem with this class in our
original definition is due to the fact that no information about the integrability
of
$$
x_1 \partial_{x_1}\varphi(x_1,\ldots,x_n),
\ldots,
x_n \partial_{x_n}\varphi(x_1,\ldots,x_n)
$$
with respect to $\mu$
is given in advance. For $\varphi\in C_0^\infty(\mathbb{R}^n)$, such functions
are bounded, hence $\mu$-integrable. In the case of an abstract locally convex space
our result shows that $X^{*}\subset L^1(\mu)$, hence $L_b\varphi\in L^1(\mu)$ for cylindrical
functions and the equation $L_b^*\mu=0$
holds also with respect to the class $\mathcal{F}\mathcal{C}_b$
in place of the original class~$\mathcal{F}\mathcal{C}_0$.

Suppose now that $\{w_n(t)\}$ is a sequence of independent Wiener processes,
$v=(v_n)_{n=1}^\infty$ is a sequence of Borel functions on $\mathbb{R}^\infty$ such that
$\sum_{n=1}^\infty |v_n(x)|^2<\infty$, and there is a diffusion process $\xi(t)=(\xi_n(t))_{n=1}^\infty$
in $\mathbb{R}^\infty$ satisfying the perturbed Ornstein--Uhlenbeck stochastic equation
$$
d\xi_n(t)=dw_n(t)-\xi_n(t)dt +v_n(\xi(t))dt.
$$
It follows from our result that if $\xi(t)$ has a stationary measure $\mu$
 for which $|v|_{l^2}\in L^1(\mu)$, then $\mu$ is absolutely
continuous with respect to the Gaussian measure that is the stationary solution for the non-perturbed
linear equation. The assumption that $|v|_{l^2}\in L^1(\mu)$ is essential
and cannot be
replaced by the weaker condition that the components of $v$ are $\mu$-integrable (recall
that the stationary equation is meaningful with this weaker condition).
Similarly,
if we have the stochastic equation
$$
d\xi_n(t)=dw_n(t)-\beta_n \xi_n(t)dt +v_n(\xi(t))dt
$$
with some $\beta_n\ge \beta_0>0$ and $\mu$ is a stationary measure, then the $\mu$-integrability
of $|v|_{l^2}$ ensures the absolute continuity
of $\mu$ with respect to the Gaussian measure corresponding to the linear system
with $v=0$. On this direction, see, e.g.,~\cite{DaP04}.

\vskip .1in

We thank M.~R\"ockner for useful discussions. We are also grateful to the anonymous referee
for the thorough reading and corrections.

This research was supported by the Russian Science Foundation Grant 17-11-01058
at Lo\-mo\-no\-sov
Moscow State University.


\begin{thebibliography}{}

\bibitem{AF}
Ambrosio, L., Figalli, A.:
Surface measures and convergence
of the Ornstein--Uhlenbeck semigroup in Wiener spaces.
Ann. Fac. Sci. Toulouse Math. (6) 20:2, 407--438 (2011)


\bibitem{AG}
Ambrosio, L.,  Gigli, N.:
A user's guide to optimal transport.
Lecture Notes in Math. V.~2062, 1--155 (2013)

\bibitem{Am}
  Ambrosio, L.,     Miranda (jr.), M.,   Maniglia, S.,    Pallara, D.:
BV functions in abstract Wiener spaces. J.~Funct. Anal. 258, 785--813 (2010)

\bibitem{ATW}
 Arnaudon, M.,  Thalmaier, A. ,  Wang, F.-Y.:
  Harnack inequality and heat kernel
estimates on manifolds with curvature unbounded below. Bull. Sci. Math. 130:3,
223--233 (2006)

\bibitem{B98}
 Bogachev, V.I.:
 Gaussian measures. Amer. Math. Soc., Providence, Rhode Island (1998)

\bibitem{B10}
Bogachev, V.I.: Differentiable measures and the Malliavin calculus.
Amer. Math. Soc., Providence, Rhode Island (2010)

\bibitem{B18}
Bogachev, V.I.: Weak convergence of measures.
Amer. Math. Soc., Providence, Rhode Island (2018)

\bibitem{B-umn18}
Bogachev, V.I.: The Ornstein--Uhlenbeck operators and semigroups.
Uspehi Matem. Nauk 73:2, 3--74 (2018)
(in Russian); English
transl.: Russian Math. Surveys 73:2 (2018)

\bibitem{BK}
Bogachev, V.I., Kolesnikov, A.V.: The Monge--Kantorovich problem: achievements,
connections,
and perspectives. Uspehi Matem. Nauk 67:5, 3--110 (2012)
(in Russian); English
transl.: Russian Math. Surveys 67:5, 785--890 (2012)

\bibitem{BKP}
Bogachev, V.I., Kosov, E.D., Popova, S.N.:
A new approach to  Nikolskii--Besov classes.
    arXiv:1707.06477 (2017).

\bibitem{BKZ}
Bogachev, V.I., Kosov, E.D., Zelenov, G.I.:
Fractional smoothness of distributions of polynomials
and a fractional analog of the Hardy--Landau--Littlewood
inequality. Trans. Amer. Math. Soc. 370:6, 4401--4432 (2018)

\bibitem{BKR01}
Bogachev, V.I., Krylov, N.V., R\"ockner, M.:
On regularity of  transition probabilities and
invariant measures of singular diffusions under minimal conditions.
Comm. Partial Differ. Equ. 26:11-12, 2037--2080 (2001)

\bibitem{BKRS}
Bogachev, V.I.,  Krylov, N.V., R\"ockner, M., Shaposhnikov, S.V.:
Fokker--Planck--Kolmogorov equations. Amer. Math. Soc., Rhode Island,
Providence (2015)

\bibitem{BMS}
Bogachev, V.I.,  Miftakhov, A.F., Shaposhnikov, S.V.:
Differential properties of
semigroups and estimates of distances between stationary distributions of diffusions.
 Doklady Akademii Nauk 485:4, 399--404 (2019) (in Russian); English
transl.: Doklady Math. 99:2, 175--180 (2019)

\bibitem{BPS}
 Bogachev, V.I., Popova, S.N., Shaposhnikov, S.V.:
 On $L^1$-estimates for probability solutions to Fokker--Planck--Kolmogorov equations.
     arXiv:1803.04568  (2018) (to appear in J. Evol. Equ.)

\bibitem{BRW01}
Bogachev, V.I., R\"ockner, M., Wang, F.-Y.:
Elliptic equations for invariant measures on finite and infinite
dimensional manifolds. J. Math. Pures Appl. 80, 177--221  (2001)

\bibitem{BR95}
 Bogachev, V.I., R\"ockner, M.: Regularity of invariant measures on
finite and infinite dimensional spaces and applications.
J.~Funct. Anal. 133, 168--223 (1995)

\bibitem{BS15}
Bogachev, V.I., Shaposhnikov, A.V.:
 Lower bounds for the Kantorovich distance.
Dokl. Akad. Nauk 460:6, 631--633 (2015) (in Russian);
English transl.: Doklady Math. 91:1, 91--93 (2015)

\bibitem{BWS15}
Bogachev, V.I., Wang, F.-Y., Shaposhnikov, A.V.:
Estimates of the Kantorovich norm on manifolds.
Dokl. Akad. Nauk  463:6, 633--638 (2015) (in Russian);
English transl.:  Doklady Math. 92:1, 1--6 (2015)

\bibitem{BWS16}
Bogachev, V.I., Wang, F.-Y., Shaposhnikov, A.V.:
On inequalities relating the Sobolev and Kantorovich norms.
Dokl. Akad. Nauk 468:2, 131--133 (2016) (in Russian);
English transl.:  Doklady Math. 93:3, 256--258 (2016)

\bibitem{DaP04}
 Da Prato, G.:   Kolmogorov equations for stochastic PDEs.
       Birkh\"auser, Basel (2004)

\bibitem{FH}
     Fukushima, M., Hino, M.:
On the space of $BV$ functions and a related stochastic
calculus in infinite dimensions. J. Funct. Anal. 183:1, 245--268 (2001)

\bibitem{KR}
Krasnosel'ski{\u{\i}}, M.A., Ruticki{\u{\i}}, Ja.B.:
Convex functions and Orlicz spaces. Noordhoff, Groningen (1961)

\bibitem{Led}
Ledoux, M.:
Isoperimetry and Gaussian analysis.
Lecture Notes in Math., V.~1648, pp.~165--294 (1996)

\bibitem{Leh}
Lehec, J.:
Regularization in $L_1$ for the Ornstein--Uhlenbeck semigroup.
Annales Facult\'e des Sci. Toulouse. Math. 25:1, 191--204 (2016)

\bibitem{RW}
R\"ockner, M., Wang, F.-Y.:
Log-Harnack inequality for stochastic
differential equations in Hilbert spaces and its consequences.
Infin. Dimen. Anal., Quantum Probab. Relat. Topics 13:1, 27--37 (2010)

\bibitem{T}
Talagrand, M.:
A conjecture on convolution operators, and a non-Dunford-Pettis operator on~$L^1$.
Israel J. Math. 68:1, 82--88 (1989)

\bibitem{V}
Villani, C.: Topics in optimal transportation.
Amer. Math. Soc., Rhode Island (2003)

\bibitem{W97}
 Wang, F.-Y.:
Logarithmic Sobolev inequalities on noncompact Riemannian manifolds.
Probab. Theory Relat. Fields 109, 417--424 (1997)

\bibitem{W1}
 Wang, F.-Y.:
Harnack inequalities on manifolds with boundary and applications.
J. Math. Pures Appl. 94, 304--321  (2010)

\bibitem{W13}
 Wang, F.-Y.:
Harnack inequalities for stochastic partial differential equations.
Springer, New York (2013)

\bibitem{W14}
Wang, F.-Y.:
Analysis for diffusion processes on Riemannian manifolds.
World Sci., Singapore (2014)

\end{thebibliography}
\end{document}